\documentclass[draft]{compositio}
\usepackage{amsmath, amscd, amssymb, mathrsfs}
\usepackage[all]{xy}
\SelectTips{cm}{}

\newtheorem{theorem}{Theorem}
\newtheorem{proposition}[theorem]{Proposition}
\newtheorem{lemma}[theorem]{Lemma}
\newtheorem{corollary}[theorem]{Corollary}
 
\theoremstyle{definition}
\newtheorem{definition}[theorem]{Definition}

\newtheorem{remark}[theorem]{Remark}

\def\inc{\operatorname{inc}}
\def\Proj{\operatorname{Proj}}

\def\can{\operatorname{can}}

\def\K{\mathbf{K}}
\def\sing{\operatorname{Sg}}
\def\Hom{\operatorname{Hom}}

\def\Ker{\operatorname{Ker}}
\def\Coker{\operatorname{Coker}}
\def\Spec{\operatorname{Spec}}
\def\straightK{\operatorname{K}}
\def\straightC{\operatorname{C}}

\def\Add{\operatorname{Add}}
\def\Flat{\operatorname{Flat}}
\def\Proj{\operatorname{Proj}}
\def\pac{\operatorname{pac}}

\begin{document}

\def\dbsing{\mathbf{D}_{\sing}}
\def\dbsinge{\mathbf{D}'_{\sing}}
\newcommand{\koszul}[2]{\operatorname{E}[#1]}
\newcommand{\ukoszul}[2]{\straightK[#1]}
\newcommand{\cech}[2]{\check{\straightC}[#1]}
\newcommand{\cat}[1]{\mathcal{#1}}
\newcommand{\qder}[1]{\mathbf{D}(#1)}
\newcommand{\qderl}[2]{\mathbf{D}_{#1}(#2)}
\newcommand{\qderu}[2]{\mathbf{D}^{#1}(#2)}
\newcommand{\lto}{\longrightarrow}
\newcommand{\rto}{\longleftarrow}
\newcommand{\xlto}[1]{\stackrel{#1}\lto}
\newcommand{\xrto}[1]{\stackrel{#1}\rto}
\newcommand{\mf}[1]{\mathfrak{#1}}
\newcommand{\add}[1]{\Add_\alpha(#1^{\textrm{op}}, \Ab)}
\newcommand{\md}[1]{\mathscr{#1}}
\newcommand{\kprof}[1]{\mathbf{N}(\Flat #1)}
\newcommand{\kflat}[1]{\K(\Flat #1)}
\newcommand{\kflatl}[2]{\K_{#1}(\Flat #2)}
\newcommand{\kproj}[1]{\K(\Proj #1)}
\newcommand{\kproju}[2]{\K^{#1}(\Proj #2)}
\newcommand{\kprofl}[2]{\mathbf{N}_{#1}(\Flat #2)}
\def\l{\,|\,}
\def\rdev{\mathbb{R}}
\def\rdevhom{\mathbb{R}\!\Hom}
\def\lotimes{\otimes^{\mathbb{L}}}
\def\libc{\bold{LC}}
\def\libk{\bold{LK}}
\def\p{\bold{p}}
\def\Ab{\cat{A}b}
\def\cand{\cat{T}^{|\alpha|}}

\title{Rouquier's Cocovering Theorem and Well-generated Triangulated Categories}
\author{Daniel Murfet}
\email{murfet@math.uni-bonn.de}
\address{Hausdorff Center for Mathematics, University of Bonn}

\begin{abstract} We study cocoverings of triangulated categories, in the sense of Rouquier, and prove that for any regular cardinal $\alpha$ the condition of $\alpha$-compactness, in the sense of Neeman, is local with respect to such cocoverings. This was established for ordinary compactness by Rouquier. Our result yields a new technique for proving that a given triangulated category is well-generated. As an application we describe the $\alpha$-compact objects in the unbounded derived category of a quasi-compact semi-separated scheme.
\end{abstract}

\maketitle

\section{Introduction}

Let $\cat{T}$ be a triangulated category with coproducts, and recall that an object $Y$ of $\cat{T}$ is \emph{compact} if the functor $\cat{T}(Y,-)$ commutes with coproducts. When $\cat{T} = \qder{X}$ is the unbounded derived category of quasi-coherent sheaves on a reasonable scheme $X$, the condition of compactness in $\cat{T}$ is local: given an open cover $\{ U_1,\ldots,U_n \}$ of $X$, an object $\md{F}$ is compact in $\qder{X}$ if and only if $\md{F}|_{U_i}$ is compact in $\qder{U_i}$ for $1 \le i \le n$. For arbitrary $\cat{T}$, Rouquier introduces in \cite[\S 5]{Rouquier08} a suitable generalisation: he defines a \emph{cocovering} of $\cat{T}$ to be a special family of Bousfield subcategories $\cat{F} = \{ \cat{I}_1,\ldots, \cat{I}_n \}$ (the precise definition is recalled below). The analogue of restriction to $U_i$ is then passage to the quotient $\cat{T} \lto \cat{T}/\cat{I}_i$, and under some natural hypotheses on $\cat{F}$, compactness in $\cat{T}$ is local: an object $Y$ is compact in $\cat{T}$ if and only if the image of $Y$ is compact in $\cat{T}/\cat{I}_i$ for $1 \le i \le n$.

This article concerns the large cardinal generalisation. Let $\alpha$ be a regular cardinal, that is, $\alpha$ is an infinite cardinal which is not the sum of fewer than $\alpha$ cardinals, all smaller than $\alpha$. In his book \cite{NeemanBook} Neeman associates to $\alpha$ a class $\cat{T}^\alpha \subseteq \cat{T}$ of \emph{$\alpha$-compact} objects. The definition is not so easily stated, but in typical examples, say the homotopy category of spectra or the derived category of an associative ring, the condition of $\alpha$-compactness is very natural; see Section \ref{section:derived_cat_rings}. In particular, the $\aleph_0$-compact objects are precisely the compact objects. Our main theorem says, among other things, that the condition of $\alpha$-compactness is local: given a cocovering $\cat{F}$ of $\cat{T}$ as above, satisfying some natural hypotheses, an object $Y$ is $\alpha$-compact in $\cat{T}$ if and only if the image of $Y$ is $\alpha$-compact in $\cat{T}/\cat{I}_i$ for $1 \le i \le n$.\\

In order to give the precise statements, we need some notation: recall that a \emph{localising subcategory} $\cat{S}$ of $\cat{T}$ is a triangulated subcategory closed under small coproducts, and $\cat{S}$ is \emph{Bousfield} if the inclusion $\cat{S} \lto \cat{T}$ has a right adjoint. Given a class $S$ of objects in $\cat{T}$, we write $\langle S \rangle$ for the smallest localising subcategory of $\cat{T}$ containing $S$. Let $\alpha$ be a regular cardinal. If $\cat{T}^\alpha$ is essentially small and $\langle \cat{T}^\alpha \rangle = \cat{T}$, then $\cat{T}$ is said to be \emph{$\alpha$-compactly generated}, and $\cat{T}$ is called \emph{well-generated} if it is $\alpha$-compactly generated for some regular cardinal $\alpha$. If $\cat{T}$ is $\alpha$-compactly generated, then a localising subcategory $\cat{S} \subseteq \cat{T}$ is \emph{$\alpha$-compactly generated in $\cat{T}$} if there is a set $S \subseteq \cat{T}^\alpha$ such that $\cat{S} = \langle S \rangle$. In this case $\cat{S}$ is $\alpha$-compactly generated, and $\cat{S}^\alpha = \cat{S} \cap \cat{T}^\alpha$ (see Theorem \ref{theorem:thomason_neeman}).

Two Bousfield subcategories $\cat{I}_1, \cat{I}_2$ of $\cat{T}$ are said to \emph{intersect properly} if, for every pair $I \in \cat{I}_1$ and $J \in \cat{I}_2$, any morphism $I \lto J$ or $J \lto I$ factors through an object of $\cat{I}_1 \cap \cat{I}_2$ \cite[(5.2.3)]{Rouquier08}. Finally, a \emph{cocovering} of $\cat{T}$ is a finite family of Bousfield subcategories $\cat{F} = \{ \cat{I}_1,\ldots, \cat{I}_n \}$ of $\cat{T}$ which are pairwise properly intersecting, such that $\bigcap_{i=1}^n \cat{I}_i = 0$; see \cite[(5.3.3)]{Rouquier08}. The $\alpha = \aleph_0$ case of the following theorem is the aforementioned result of Rouquier, namely \cite[Theorem 5.15]{Rouquier08}. 

\begin{theorem}\label{theorem:main_theorem_general} Let $\cat{T}$ be a triangulated category with coproducts and $\alpha$ a regular cardinal. Suppose that $\cat{F} = \{ \cat{I}_1,\ldots, \cat{I}_n \}$ is a cocovering of $\cat{T}$ with the following properties:
\begin{itemize}
\item[(1)] $\cat{T}/\cat{I}$ is $\alpha$-compactly generated for every $\cat{I} \in \cat{F}$.
\item[(2)] For every $\cat{I} \in \cat{F}$ and nonempty subset $\cat{F}' \subseteq \cat{F} \setminus \{ \cat{I} \}$ the essential image of the composite
\[
\bigcap_{\cat{I}' \in \cat{F}'} \cat{I}' \xlto{\inc} \cat{T} \xlto{\can} \cat{T}/\cat{I}
\]
is $\alpha$-compactly generated in $\cat{T}/\cat{I}$.
\end{itemize}
Then $\cat{T}$ is $\alpha$-compactly generated, and an object $X \in \cat{T}$ is $\alpha$-compact if and only if the image of $X$ is $\alpha$-compact in $\cat{T}/\cat{I}$ for every $\cat{I} \in \cat{F}$. Let $\cat{S}$ be a Bousfield subcategory of $\cat{T}$ intersecting properly with each $\cat{I} \in \cat{F}$, such that:
\begin{itemize}
\item[(3)] $\cat{S}/(\cat{S} \cap \cat{I})$ is $\alpha$-compactly generated in $\cat{T}/\cat{I}$ for every $\cat{I} \in \cat{F}$.
\item[(4)] For every $\cat{I} \in \cat{F}$ and nonempty subset $\cat{F}' \subseteq \cat{F} \setminus \{ \cat{I} \}$ the essential image of the composite
\[
\cat{S} \cap \bigcap_{\cat{I}' \in \cat{F}'} \cat{I}' \xlto{\inc} \cat{T} \xlto{\can} \cat{T}/\cat{I}
\]
is $\alpha$-compactly generated in $\cat{T}/\cat{I}$.
\end{itemize}
Then $\cat{S}$ is $\alpha$-compactly generated in $\cat{T}$.
\end{theorem}

To return to the geometric example: if $\cat{T} = \qder{X}$ and we are given an open cover as above, then for each $1 \le i \le n$ denote by $\cat{I}_i = \qderl{X \setminus U_i}{X}$ the full subcategory of $\qder{X}$ consisting of complexes with cohomology supported on $X \setminus U_i$. There is a canonical equivalence $\cat{T}/\cat{I}_i \cong \qder{U_i}$, the quotient functor $\cat{T} \lto \cat{T}/\cat{I}_i$ corresponds to restriction, and the family $\cat{F} = \{ \cat{I}_1,\ldots,\cat{I}_n \}$ is a cocovering of $\qder{X}$ satisfying the hypotheses $(1),(2)$ of the theorem for $\alpha = \aleph_0$ \cite[\S 6.2]{Rouquier08}. For this choice of $\cat{T}$ and $\cat{F}$ the hypotheses are very natural, and easily verified; for the full elaboration, see Section \ref{example:compacts_derived}.

 Applying the theorem (recall that, since $\alpha = \aleph_0$, this is just Rouquier's \cite[Theorem 5.15]{Rouquier08}) one obtains a proof of the fact, due originally to Neeman \cite{Neeman96}, that the compact objects in $\qder{X}$ are precisely the perfect complexes; see \cite[Theorem 6.8]{Rouquier08}. Using the $\alpha > \aleph_0$ case of the theorem we obtain in Section \ref{example:compacts_derived} a description of the $\alpha$-compact objects in $\qder{X}$.

We have another application in mind, which will appear in the forthcoming \cite{MurfetPure}. Let $A$ be an associative ring with identity, $\kproj{A}$ and $\kflat{A}$ the homotopy categories of projective and flat left $A$-modules, respectively. A complex of left $A$-modules $F$ is \emph{pure acyclic} if it is acyclic, and $N \otimes_A F$ is acyclic for every right $A$-module $N$. Let $\kflatl{\pac}{A}$ denote the full subcategory of pure acyclic complexes in $\kflat{A}$. This is a triangulated subcategory, and Neeman proves in \cite{Neeman08} that the composite
\begin{equation}\label{eq:neeman_equivalence}
\kproj{A} \xlto{\inc} \kflat{A} \xlto{\can} \kflat{A}/\kflatl{\pac}{A}
\end{equation}
is an equivalence. Now let $X$ be a quasi-compact semi-separated scheme. Unless $X$ is affine, projective quasi-coherent sheaves on $X$ are rare, and the homotopy category of projective quasi-coherent sheaves on $X$ is often the zero category. In this case, the equivalence (\ref{eq:neeman_equivalence}) suggests a suitable replacement. Let $\kflat{X}$ be the homotopy category of flat quasi-coherent sheaves on $X$, and denote by $\kflatl{\pac}{X}$ the full subcategory of acyclic complexes $\md{F}$ with the property that $\md{F} \otimes_{\cat{O}_X} \md{A}$ is acyclic for every quasi-coherent sheaf $\md{A}$. Define
\[
\kprof{X} := \kflat{X}/\kflatl{\pac}{X},
\]
and let $\{ U_1, \ldots, U_n \}$ be an affine open cover of $X$, with say $U_i \cong \Spec(A_i)$ for $1 \le i \le n$. We show in \cite{MurfetPure} that there is a cocovering of $\kprof{X}$ by Bousfield subcategories $\{ \kprofl{X \setminus U_i}{X} \}_{1 \le i \le n}$, where $\kprofl{X \setminus U_i}{X}$ is the kernel of a natural restriction functor $\kprof{X} \lto \kprof{U_i}$. Moreover, there are canonical equivalences
\[
\kprof{X}/\kprofl{X \setminus U_i}{X} \cong \kprof{U_i} \cong \kproj{A_i}.
\]
Neeman proves in \emph{loc.cit.}\ that $\kproj{A_i}$ is $\aleph_1$-compactly generated, and even compactly generated when $A_i$ is coherent.  In \cite{MurfetPure} we combine Neeman's results with Theorem \ref{theorem:main_theorem_general} to see that the global category $\kprof{X}$ is $\aleph_1$-compactly generated, and compactly generated when $X$ is noetherian.

The proof of Theorem \ref{theorem:main_theorem_general} is by induction on the size $n = |\cat{F}|$ of the cocovering. The real content is in the initial step of the induction, which we separate into its own section. The proof of the theorem is completed in Section \ref{section:proof}. Our basic reference for triangulated categories is \cite{NeemanBook}, whose notation we follow with one exception: given a class $\cat{C}$ of objects in $\cat{T}$, we write
\begin{align*}
\cat{C}^{\perp} &= \{ Y \in \cat{T} \l \Hom_{\cat{T}}(\Sigma^n X, Y) = 0 \text{ for all } X \in \cat{C} \text{ and } n \in \mathbb{Z} \},\\
{}^{\perp} \cat{C} &= \{ X \in \cat{T} \l \Hom_{\cat{T}}(X, \Sigma^n Y) = 0 \text{ for all } Y \in \cat{C} \text{ and } n \in \mathbb{Z} \}
\end{align*}
for the orthogonals, which are triangulated subcategories of $\cat{T}$. For further information on the theory of well-generated triangulated categories, the reader is referred to \cite{Neeman05,Krause08}. In this article, all triangulated categories have ``small Homs''.
\\

\emph{Acknowledgements.} I would like to thank Amnon Neeman for suggesting improvements to an earlier version and communicating the proof of Theorem \ref{theorem:neeman_compact_cohomology}, and Henning Krause for helpful discussion on the subject of this paper.

\section{Initial step of the induction}

Throughout, $\cat{T}$ is a triangulated category with coproducts. Two Bousfield subcategories $\cat{I}_1,\cat{I}_2$ of $\cat{T}$ are \emph{orthogonal} if $\cat{I}_1 \subseteq \cat{I}_2^{\perp}$ and $\cat{I}_2 \subseteq \cat{I}_1^{\perp}$. In this situation the composite $\cat{I}_a \lto \cat{T} \lto \cat{T}/\cat{I}_b$ is fully faithful for $\{ a, b \} = \{1,2\}$ and $\cat{I}_a$ may be identified with a Bousfield subcategory of $\cat{T}/\cat{I}_b$. Let us state the $n = 2$ case of the Theorem \ref{theorem:main_theorem_general} as a proposition:

\begin{proposition}\label{prop:main_theorem_1} Let $\cat{I}_1, \cat{I}_2$ be orthogonal Bousfield subcategories of $\cat{T}$, and suppose that for some regular cardinal $\alpha$, we have:
\begin{itemize}
\item[(1)] $\cat{T}/\cat{I}_a$ is $\alpha$-compactly generated for $a \in \{ 1,2 \}$,
\item[(2)] $\cat{I}_a$ is $\alpha$-compactly generated in $\cat{T}/\cat{I}_b$ for $\{a,b\} = \{1,2\}$.
\end{itemize}
Then $\cat{T}$ is $\alpha$-compactly generated, and an object $X \in \cat{T}$ is $\alpha$-compact if and only if the image of $X$ is $\alpha$-compact in both $\cat{T}/\cat{I}_1$ and $\cat{T}/\cat{I}_2$. Let $\cat{S}$ be a Bousfield subcategory of $\cat{T}$ intersecting properly with $\cat{I}_1$ and $\cat{I}_2$, and suppose that:
\begin{itemize}
\item[(3)] $\cat{S}/(\cat{S} \cap \cat{I}_a)$ is $\alpha$-compactly generated in $\cat{T}/\cat{I}_a$ for $a \in \{1,2\}$,
\item[(4)] $\cat{S} \cap \cat{I}_a$ is $\alpha$-compactly generated in $\cat{T}/\cat{I}_b$ for $\{a,b\} = \{ 1,2 \}$.
\end{itemize}
Then $\cat{S}$ is $\alpha$-compactly generated in $\cat{T}$.
\end{proposition}

We develop the proof as a series of lemmas. Since the $\alpha = \aleph_0$ is handled by \cite[Proposition 5.14]{Rouquier08}, we assume that $\alpha > \aleph_0$. Throughout this section the notation of the proposition is in force. For $a \in \{1,2\}$ we write ${i_a}_*: \cat{I}_a \lto \cat{T}$ for the inclusion, $i_a^!$ for the right adjoint of ${i_a}_*$, $j_a^*: \cat{T} \lto \cat{T}/\cat{I}_a$ for the quotient functor and ${j_a}_*$ for the right adjoint of $j_a^*$.

To prove that $\cat{T}$ is $\alpha$-compactly generated we need to produce, in the language of \cite[Ch.8]{NeemanBook}, an $\alpha$-perfect set of $\alpha$-small objects, which generates. The condition of $\alpha$-smallness is very simple: an object $X \in \cat{T}$ is \emph{$\alpha$-small} if for every family $\{ Y_i \}_{i \in I}$ of objects of $\cat{T}$, any morphism
\[
X \lto \bigoplus_{i \in I} Y_i
\]
factors through a subcoproduct $\bigoplus_{i \in J} Y_i$ for some subset $J \subseteq I$ of cardinality $|J| < \alpha$. An object $X$ is $\aleph_0$-small if and only if $\cat{T}(X,-)$ commutes with coproducts, and in this case one says that $X$ is \emph{compact}. We refer the reader to \cite[Ch.3]{NeemanBook} for the definition of $\alpha$-perfect classes, and restrict ourselves here to one trivial fact: any triangulated subcategory of $\cat{T}$ is an $\aleph_0$-perfect class.

By hypothesis the quotients $\cat{T}/\cat{I}_a$ are $\alpha$-compactly generated, and $\cat{I}_b$ is $\alpha$-compactly generated in $\cat{T}/\cat{I}_a$ for $\{a,b\} = \{1,2\}$. Hence these categories all possess $\alpha$-perfect classes of $\alpha$-small objects which generate. The strategy employed by Rouquier \cite{Rouquier08} in the $\alpha = \aleph_0$ case is to take generating sets $\cat{E}$ and $\cat{E}'$ for $\cat{I}_2, \cat{T}/\cat{I}_2$ respectively, use a gluing argument to lift $\cat{E}'$ to class $\cat{E}''$ of compact objects in $\cat{T}$, and take the union $\cat{E} \cup \cat{E}''$. This is a generating set of compact objects for $\cat{T}$.

In the $\alpha > \aleph_0$ case we take a different approach, in which it seems easier to manage the perfection condition (which is trivial for $\alpha = \aleph_0$). To proceed, we first recall how to rephrase the condition on our generating set in terms of a property of a certain exact functor between abelian categories.

A triangulated subcategory $\cat{S} \subseteq \cat{T}$ is said to be \emph{$\alpha$-localising} if the coproduct of fewer than $\alpha$ objects of $\cat{S}$ lies in $\cat{S}$. For example, the class $\cat{T}^\alpha$ of $\alpha$-compact objects is an $\alpha$-localising triangulated subcategory of $\cat{T}$. Given an $\alpha$-localising subcategory $\cat{S}$ of $\cat{T}$ we denote by $\add{\cat{S}}$ the abelian category of all functors $\cat{S}^{\textrm{op}} \lto \Ab$ which preserve products of fewer than $\alpha$ objects, where $\Ab$ is the category of abelian groups. There is a canonical homological functor
\[
\cat{T} \lto \add{\cat{S}}, \qquad X \mapsto \cat{T}(-,X)|_{\cat{S}}.
\]
Let $\cat{T} \lto A(\cat{T})$ be Freyd's universal homological functor, where $A(\cat{T})$ is the abelianisation of $\cat{T}$ \cite[Ch. 5]{NeemanBook}. From the universal property of this construction, we deduce an exact functor
\[
\pi: A(\cat{T}) \lto \add{\cat{S}},
\]
and Neeman proves in \cite[Theorem 1.8]{NeemanBook} that $\cat{S}$ is an $\alpha$-perfect class of $\alpha$-small objects precisely when $\pi$ preserves coproducts. Here, then, is the strategy of our proof: in Definition \ref{definition:generating_set} below we take the obvious candidate for a generating set $\cat{S} = \cat{T}^{|\alpha|}$ of $\cat{T}$. We have to prove two things: firstly, that this is an $\alpha$-perfect class of $\alpha$-small objects, and secondly, that it generates. The second condition is easily verified, and for the first we just need to prove that $\pi$ preserves coproducts. The cocovering $\{ \cat{I}_1, \cat{I}_2 \}$ of $\cat{T}$ leads to a pair of localisations of $\add{\cat{S}}$, which we may think of as a ``cover'' of this abelian category. Checking that $\pi$ preserves coproducts then becomes a ``local'' problem with respect to this cover. The local pieces in the cover correspond to the quotients $\cat{T}/\cat{I}_a$, and we can use the fact that these categories are $\alpha$-compactly generated to complete the proof.

\begin{definition}\label{definition:generating_set} We define a full subcategory of $\cat{T}$ by
\[
\cand = \{ X \in \cat{T} \l j_a^*(X) \in (\cat{T}/\cat{I}_a)^\alpha \text{ for } a \in \{ 1,2 \} \}.
\]
\end{definition}

\begin{lemma} $\cand$ is an $\alpha$-localising subcategory of $\cat{T}$.
\end{lemma}
\begin{proof}
Follows from the fact that $j_a^*$ preserve coproducts, and $(\cat{T}/\cat{I}_a)^\alpha$ is $\alpha$-localising.
\end{proof}

Let us recall the statement of the Neeman-Ravenel-Thomason localisation theorem.

\begin{theorem}\label{theorem:thomason_neeman} Let $\cat{R}$ be a trangulated category with coproducts which is $\alpha$-compactly generated, and let $\cat{S} \subseteq \cat{R}$ be a localising subcategory $\alpha$-compactly generated in $\cat{R}$. Then $\cat{S}$ is $\alpha$-compactly generated, and $\cat{S}^\alpha = \cat{R}^\alpha \cap \cat{S}$. The canonical functor $\cat{R} \lto \cat{R}/\cat{S}$ preserves $\alpha$-compactness and the induced functor
\[
\cat{R}^\alpha/\cat{S}^\alpha \lto (\cat{R}/\cat{S})^\alpha
\]
is an equivalence \textup{(}recall that $\alpha > \aleph_0$\textup{)}.
\end{theorem}
\begin{proof}
See \cite[Theorem 4.4.9]{NeemanBook}.
\end{proof}

The full subcategory of $\alpha$-compact objects in $\cat{I}_a$ is denoted by $\cat{I}_a^\alpha$. One needs to be careful to distinguish between objects $X \in \cat{I}_a$ which are $\alpha$-compact in $\cat{I}_a$, and those that are $\alpha$-compact in the larger category $\cat{T}$. At this point, we do not know that these classes are the same. It follows from hypotheses $(1)$ and $(2)$ of Proposition \ref{prop:main_theorem_1}, and Theorem \ref{theorem:thomason_neeman}, that $\cat{I}_a^\alpha \subseteq \cat{I}_a$ is precisely the class of objects $X \in \cat{I}_a$ with the property that $j_b^*(X) \in (\cat{T}/\cat{I}_b)^\alpha$, where $\{a,b\} = \{1,2\}$. Moreover, $\cat{I}_a = \langle \cat{I}_a^\alpha \rangle$.

\begin{lemma} There is an inclusion $\cat{I}_1^\alpha \cup \cat{I}_2^\alpha \subseteq \cand$.
\end{lemma}
\begin{proof}
If $X \in \cat{I}_a^\alpha$ then by $(2)$, $j_b^*(X) \in (\cat{T}/\cat{I}_b)^\alpha$. Since $j_a^*(X) = 0$, it follows that $X \in \cand$.
\end{proof}

\begin{lemma}\label{lemma:factorisation} Given $X \in \cand$ and $Y \in \cat{I}_a$ for $a \in \{1,2\}$, any morphism $f: X \lto Y$ in $\cat{T}$ factors as
\[
X \lto I \lto Y
\]
for some $I \in \cat{I}_a^\alpha$.
\end{lemma}
\begin{proof}
Let $b \in \{1,2\}$ be such that $b \neq a$. By hypothesis $(2)$ there is a set $Q \subseteq (\cat{T}/\cat{I}_b)^\alpha \cap \cat{I}_a$ such that $\langle Q \rangle = \cat{I}_a$. By \cite[Theorem 4.3.3]{NeemanBook} the morphism $j_b^*f: j_b^*X \lto j_b^*Y$ factors in $\cat{T}/\cat{I}_b$ as
\[
j_b^*X \lto N \lto j_b^*Y
\]
for some $N \in \langle Q \rangle^\alpha = \cat{I}_a^\alpha$. Since $I := {j_b}_*N$ belongs to $\cat{I}_a^\alpha$, the composite
\[
X \xlto{\can} {j_b}_*j_b^*X \lto {j_b}_*N \lto {j_b}_*j_b^*Y \cong Y
\]
provides the desired factorisation of $f$.
\end{proof}

We use several facts about proper intersection of subcategories developed by Rouquier \cite[\S 5]{Rouquier08}. For the reader's convenience, the necessary facts are recalled here in Appendix \ref{section:proper_intersection}. For example, since $\cat{I}_1, \cat{I}_2$ are properly intersecting the Verdier sum operation is commutative: $\cat{I}_1 \star \cat{I}_2 = \cat{I}_2 \star \cat{I}_1$. It follows that $\cat{I}_1 \star \cat{I}_2$ is a Bousfield subcategory of $\cat{T}$ and, following the notation of \cite[\S 5]{Rouquier08}, we write ${i_{\cup}}_*: \cat{I}_1 \star \cat{I}_2 \lto \cat{T}$ for the inclusion, $i_{\cup}^!$ for its right adjoint, $j_{\cup}^*: \cat{T} \lto \cat{T}/(\cat{I}_1 \star \cat{I}_2)$ for the quotient, and ${j_{\cup}}_*$ for its right adjoint. Note that in \emph{loc.cit.}\ Rouquier writes $\langle \cat{I}_1 \cup \cat{I}_2 \rangle_{\infty}$ for $\cat{I}_1 \star \cat{I}_2$, to reflect the fact that this is the smallest triangulated subcategory of $\cat{T}$ containing $\cat{I}_1 \cup \cat{I}_2$. For $\{a,b\} = \{1,2\}$ the quotient $j_{\cup}^*$ induces a functor $j_{a\cup}^*: \cat{T}/\cat{I}_a \lto \cat{T}/(\cat{I}_1 \star \cat{I}_2)$ fitting into a sequence
\[
0 \lto \cat{I}_b \lto \cat{T}/\cat{I}_a \lto \cat{T}/(\cat{I}_1 \star \cat{I}_2) \lto 0
\]
which is \emph{exact}, in the sense that $\cat{I}_b \lto \cat{T}/\cat{I}_a$ is fully faithful and $j_{a \cup}^*$ is, up to natural equivalence, the Verdier quotient of $\cat{T}/\cat{I}_a$ by $\cat{I}_b$. We write ${j_{a\cup}}_*$ for the right adjoint of $j_{a\cup}^*$.

\begin{lemma}\label{lemma:lifting} Given $a \in \{1,2\}$ and $Y \in (\cat{T}/\cat{I}_a)^\alpha$, there is $X \in \cand$ such that $j_a^*(X) \cong Y$.
\end{lemma}
\begin{proof}
We use the argument given in the proof of \cite[Proposition 5.14]{Rouquier08}. Let $b \in \{1,2\}$ be such that $\{a,b\} = \{1,2\}$. From hypotheses $(1),(2)$ and Theorem \ref{theorem:thomason_neeman} we deduce that the quotient functor $j_{a \cup}^*$ preserves $\alpha$-compactness. Hence, if we set $D_a = Y$, then the object $D_{\cup} := j_{a\cup}^*D_a$ is $\alpha$-compact in $\cat{T}/(\cat{I}_1 \star \cat{I}_2)$. Also by Theorem \ref{theorem:thomason_neeman}, the canonical functor
\[
j_{b\cup}^*: (\cat{T}/\cat{I}_b)^\alpha \lto \Big( \cat{T}/(\cat{I}_1 \star \cat{I}_2) \Big)^\alpha
\]
is a Verdier quotient, so we can find $D_b \in (\cat{T}/\cat{I}_b)^\alpha$ and an isomorphism $j_{b\cup}^*D_b \cong D_{\cup}$. There are unit morphisms $\eta_1: D_1 \lto {j_{1\cup}}_*D_\cup, \eta_2: D_2 \lto {j_{2\cup}}_*D_\cup$ and we define $\delta$ to be the morphism induced out of the coproduct ${j_1}_*D_1 \oplus {j_2}_*D_2$ by ${j_1}_*(\eta_1) - {j_2}_*(\eta_2)$. If we define $X$ by extending $\delta$ to a triangle
\[
X \lto {j_1}_*D_1 \oplus {j_2}_*D_2 \xlto{\delta} {j_{\cup}}_* D_{\cup} \xlto{+},
\]
then one checks that $j_a^*X \cong D_a = Y$, and that $j_b^*X \cong D_b$, so $X \in \cand$ as required.
\end{proof}

\begin{lemma}\label{lemma:essential_smallness} $\cand$ is essentially small.
\end{lemma}
\begin{proof}
For $X \in \cand$ there is a canonical triangle \cite[Proposition 5.10]{Rouquier08}
\begin{equation}\label{eq:essential_smallness}
X \lto {j_1}_*j_1^*X \oplus {j_2}_*j_2^*X \lto {j_\cup}_*j_\cup^*X \xlto{+}.
\end{equation}
By hypothesis $j_a^*X \in (\cat{T}/\cat{I}_a)^\alpha$ for $a \in \{1,2\}$. Now, since $\cat{T}/\cat{I}_a$ is $\alpha$-compactly generated, $(\cat{T}/\cat{I}_a)^\alpha$ is essentially small. It follows that there is, up to isomorphism, only a ``set'' of possible objects $X$ in a triangle of the form (\ref{eq:essential_smallness}), whence $\cand$ is essentially small.
\end{proof}

\begin{lemma}\label{lemma:equivalence_quotients} For $a \in \{1,2\}$ the canonical functor
\begin{equation}\label{eq:equivalence_quotients}
j_a^*: \cand/(\cand \cap \cat{I}_a) \lto (\cat{T}/\cat{I}_a)^\alpha
\end{equation}
is an equivalence.
\end{lemma}
\begin{proof}
The composite $\cand \xlto{\inc} \cat{T} \xlto{j_a^*} \cat{T}/\cat{I}_a$ factors, by definition, through the inclusion $(\cat{T}/\cat{I}_a)^\alpha \lto \cat{T}/\cat{I}_a$. The factorisation $\cand \lto (\cat{T}/\cat{I}_a)^\alpha$ vanishes on $\cand \cap \cat{I}_a$, and induces a functor $(\ref{eq:equivalence_quotients})$. To verify that $(\ref{eq:equivalence_quotients})$ is fully faithful, we use a standard argument. Let $s: X \lto Y$ be a morphism in $\cat{T}$ with cone in $\cat{I}_a$ and $Y \in \cand$. Extend to a triangle
\[
X \xlto{s} Y \xlto{f} I \xlto{+}.
\]
The map $f$ factors, by Lemma \ref{lemma:factorisation}, as $Y \lto I' \lto I$ with $I' \in \cat{I}_a^\alpha$. From the octahedral axiom, applied to the pair of morphisms in this factorisation, we obtain objects $C,D$ and triangles
\begin{align}
Y \lto I' \lto C \xlto{+},\\ 
I' \lto I \lto D \xlto{+},\\ 
C \lto \Sigma X \lto D \xlto{+}.
\end{align}
Since $I' \in \cat{I}_a^\alpha \subseteq \cand$ we find that $C$ belongs to $\cand$ and $D$ to $\cat{I}_a$. Thus $\Sigma^{-1} C \lto X$ is a morphism with domain in $\cand$, the composite of which with $s: X \lto Y$ has cone in $\cat{I}_a$. It now follows easily that (\ref{eq:equivalence_quotients}) is fully faithful. To see that it is surjective on objects, we use Lemma \ref{lemma:lifting}.
\end{proof}

Fix an index $a \in \{1,2\}$. By Lemma \ref{lemma:equivalence_quotients} the canonical functor $j_a^*: \cand \lto (\cat{T}/\cat{I}_a)^\alpha$ is a Verdier quotient which preserves $\alpha$-coproducts. By \cite[Lemma B.8]{Krause08} the (exact) restriction functor
\begin{gather*}
q_a^*: \add{\{(\cat{T}/\cat{I}_a)^\alpha\}} \lto \add{\{\cand\}},\\
q_a^*(F) = F \circ j_a^*
\end{gather*}
has an exact left adjoint ${q_a}*$. The right adjoint $q_a^*$ is fully faithful, so ${q_a}*$ is a Gabriel localisation of $\add{\{\cand\}}$. As we will see, it is reasonable to think of the pair of localisations
\[
\xymatrix@C+2pc{
&  \add{\{(\cat{T}/\cat{I}_1)^\alpha\}}\\
\add{\{\cand\}} \ar@/^1.3pc/[ur]^(0.4){{q_1}_*} \ar@/_1.3pc/[dr]_(0.4){{q_2}_*}\\
& \add{\{(\cat{T}/\cat{I}_2)^\alpha\}}
}
\]
as a covering of $\add{\{\cand\}}$. To make this precise, we show that a functor $F$ which is sent to zero by both localisations, must already be zero.

\begin{lemma}\label{lemma:intersection_of_kernels} $\Ker({q_1}_*) \cap \Ker({q_2}_*) = 0$.
\end{lemma}
\begin{proof}
Fix $a \in \{1,2\}$. By \cite[Lemma B.8]{Krause08} a functor $F \in \add{\{\cand\}}$ belongs to $\Ker({q_a}_*)$ if and only if for any $C \in \cand$, every morphism $\cand(-,C) \lto F$ factors via $\cand(-,\gamma): \cand(-,C) \lto \cand(-,C')$ for some morphism $\gamma: C \lto C'$ in $\cand$ with $j_a^*(\gamma) = 0$ in $\cat{T}/\cat{I}_a$. From Lemma \ref{lemma:equivalence_quotients} we deduce that $j_a^*(\gamma) = 0$ if and only if $\gamma$ factors, in $\cand$, via an object of $\cand \cap \cat{I}_a$. We conclude that $F$ belongs to $\Ker({q_a}_*)$ if and only if every morphism $\cand(-,C) \lto F$ factors via $\cand(-,I)$ for some $I \in \cand \cap \cat{I}_a$.

Assume now that $F$ belongs to $\Ker({q_1}_*) \cap \Ker({q_2}_*)$ and let $x: \cand(-,C) \lto F$ be any morphism. By the above, this must factor as $\cand(-,C) \lto \cand(-,I) \lto F$ for some $I \in \cand \cap \cat{I}_1$. Since $F$ also belongs to $\Ker({q_2}_*)$, the morphism $\cand(-,I) \lto F$ factors as $\cand(-,I) \lto \cand(-,I') \lto F$ for some $I' \in \cand \cap \cat{I}_2$. But since $\cat{I}_1$ and $\cat{I}_2$ are orthogonal, the morphism $\cand(-,I) \lto \cand(-,I')$ vanishes, and we conclude that $x = 0$. It follows that $F = 0$, as claimed.
\end{proof}

\begin{proposition}\label{prop:cand_is_perfect} $\cand$ is an $\alpha$-perfect class of $\alpha$-small objects in $\cat{T}$.
\end{proposition}
\begin{proof} We have $\alpha$-localising subcategories $\cat{T}^{|\alpha|} \subseteq \cat{T}, (\cat{T}/\cat{I}_1)^\alpha \subseteq \cat{T}/\cat{I}_1$ and $(\cat{T}/\cat{I}_2)^\alpha \subseteq \cat{T}/\cat{I}_2$ and, as discussed at the beginning of this section, there are canonical exact functors
\begin{gather*}
\pi: A(\cat{T}) \lto \add{\{\cand\}},\\
\pi_1: A(\cat{T}/\cat{I}_1) \lto \add{\{(\cat{T}/\cat{I}_1)^\alpha\}},\\
\pi_2: A(\cat{T}/\cat{I}_2) \lto \add{\{(\cat{T}/\cat{I}_2)^\alpha\}}.
\end{gather*}
We claim that for $a \in \{1,2\}$ the diagram
\begin{equation}\label{eq:cand_is_perfect_0}
\xymatrix@C+2pc{
A(\cat{T}) \ar[d]_{\pi} \ar[r]^{A(j_a^*)} & A(\cat{T}/\cat{I}_a) \ar[d]^{\pi_a}\\
\add{\{\cand\}} \ar[r]_{{q_a}_*} & \add{\{(\cat{T}/\cat{I}_a)^\alpha\}}
}
\end{equation}
commutes up to natural equivalence, where $A(j_a^*)$ is the induced functor between the abelianisations. By the universal property of the abelianisations, it suffices to prove that the related diagram
\begin{equation}\label{eq:cand_is_perfect}
\xymatrix@C+2pc{
\cat{T} \ar[d]_{\rho} \ar[r]^{j_a^*} & \cat{T}/\cat{I}_a \ar[d]^{\rho_a}\\
\add{\{\cand\}} \ar[r]_{{q_a}_*} & \add{\{(\cat{T}/\cat{I}_a)^\alpha\}}
}
\end{equation}
commutes, where $\rho$ and $\rho_a$ are the restricted Yoneda functors. To do this, we recycle an argument of Krause from the proof of \cite[Theorem 6.3]{Krause08}. The first thing to observe is that the composite ${q_a}_* \circ \rho$ vanishes on $\cat{I}_a$: one uses the description in \cite[Lemma B.8]{Krause08} of the kernel of ${q_a}_*$, together with Lemma \ref{lemma:factorisation}. For $C \in \cand$ and $X \in \cat{T}/\cat{I}_a$ there is an adjunction isomorphism
\[
\cat{T}/\cat{I}_a(j_a^* C, X) \cong \cat{T}(C, {j_a}_* X),
\]
and it follows that there is a natural equivalence $q_a^* \circ \rho_a \cong \rho \circ {j_a}_*$. Composing with ${q_a}_*$ we obtain a natural equivalence $\rho_a \cong {q_a}_* \circ q_a^* \circ \rho_a \cong {q_a}_* \circ \rho \circ {j_a}_*$ and consequently $\rho_a \circ j_a^* \cong {q_a}_* \circ \rho \circ {j_a}_* \circ j_a^*$. From the unit $\eta: 1 \lto {j_a}_* \circ j_a^*$ we obtain a natural transformation
\[
\xymatrix@C+2pc{
{q_a}_* \circ \rho \ar[r]^(0.3){({q_a}_* \circ \rho) \eta} & {q_a}_* \circ \rho \circ {j_a}_* \circ j_a^* \cong \rho_a \circ j_a^*.
}
\]
This is the desired natural equivalence, because for every $X \in \cat{T}$ the cone of $\eta_X: X \lto {j_a}_* j_a^*(X)$ is an object of $\cat{I}_a$, on which ${q_a}_* \circ \rho$ vanishes.


Since $(\cat{T}/\cat{I}_a)^\alpha$ is an $\alpha$-perfect class of $\alpha$-small objects, we infer from \cite[Theorem 1.8]{NeemanBook} that $\pi_a$ preserves coproducts. Let $\{ x_\lambda \}_{\lambda}$ be a family of objects in $A(\cat{T})$, and let $\xi: \bigoplus_\lambda \pi(x_\lambda) \lto \pi(\bigoplus_\lambda x_\lambda)$ be the canonical morphism in $\add{\{\cand\}}$. Extend on both sides to an exact sequence
\[
0 \lto \Ker(\xi) \lto \bigoplus_\lambda \pi(x_\lambda) \xlto{\xi} \pi\Big( \bigoplus_\lambda x_\lambda \Big) \lto \Coker(\xi) \lto 0,
\]
which maps under ${q_a}_*$ to an exact sequence
\[
0 \lto {q_a}_*\Ker(\xi) \lto {q_a}_*\bigoplus_\lambda \pi(x_\lambda) \xlto{{q_a}_*(\xi)} {q_a}_*\pi\Big( \bigoplus_\lambda x_\lambda \Big) \lto {q_a}_*\Coker(\xi) \lto 0.
\]
Here is where we use commutativity of (\ref{eq:cand_is_perfect_0}). Both $\pi_a$ and $A(j_a^*)$ preserve coproducts, whence ${q_a}_* \circ \pi \cong \pi_a \circ A(j_a^*)$ preserves coproducts. Since ${q_a}_*$ preserves coproducts (it has a right adjoint), we conclude that ${q_a}_*(\xi)$ is an isomorphism, and thus ${q_a}_*\Ker(\xi)$ and ${q_a}_*\Coker(\xi)$ both vanish. Since $a \in \{1,2\}$ was arbitrary, it follows from Lemma \ref{lemma:intersection_of_kernels} that $\Ker(\xi) = \Coker(\xi) = 0$, whence $\xi$ is an isomorphism and $\pi$ preserves coproducts. By \cite[Theorem 1.8]{NeemanBook}, $\cat{T}^{|\alpha|}$ is an $\alpha$-perfect class of $\alpha$-small objects.
\end{proof}

\begin{proof}[Proof of Proposition \ref{prop:main_theorem_1}]
First we prove that $\cand$ is an $\alpha$-compact generating set\footnote{Strictly speaking $\cand$ is an essentially small class, not a set, but let us replace $\cand$ by a representative set of objects and ignore the distinction.} for $\cat{T}$, in the sense of \cite[Definition 8.1.6]{NeemanBook}. In light of Proposition \ref{prop:cand_is_perfect}, it suffices to prove that if an object $x \in \cat{T}$ satisfies $\cat{T}(y,x) = 0$ for all $y \in \cand$ then $x = 0$. Note that $\cat{I}_1^\alpha \subseteq \cand$, so $x \in (\cand)^{\perp} \subseteq (\cat{I}_1^\alpha)^{\perp} = \cat{I}_1^{\perp}$, since $\langle \cat{I}_1^\alpha \rangle = \cat{I}_1$. Let $t \in (\cat{T}/\cat{I}_1)^\alpha$ be given, and choose by Lemma \ref{lemma:lifting} a $t' \in \cand$ with $j_1^*(t') \cong t$. Then
\[
0 = \cat{T}(t', x) \cong \cat{T}/\cat{I}_1( j_1^*t', j_1^*x) \cong \cat{T}/\cat{I}_1(t, j_1^*x).
\]
But $\cat{T}/\cat{I}_1$ is $\alpha$-compactly generated and $t$ was arbitrary, so $j_1^* x = 0$. Hence $x$ belongs to both $\cat{I}_1$ and $\cat{I}_1^{\perp}$, which is only possible if $x = 0$. It now follows from \cite[Proposition 8.4.2]{NeemanBook} that $\cat{T} = \langle \cand \rangle$, and from \cite[Theorem 4.4.9]{NeemanBook} that $\cat{T}^\alpha$ is the smallest $\alpha$-localising subcategory of $\cat{T}$ containing $\cand$. Hence $\cat{T}^\alpha = \cand$, which settles the first statement of the theorem. 

The second statement of the theorem deals with a Bousfield subcategory $\cat{S}$. The intersections $\cat{S} \cap \cat{I}_1, \cat{S} \cap \cat{I}_2$ are orthogonal Bousfield subcategories of $\cat{S}$. We want to apply the first part of the theorem to $\cat{S}$ and this pair of subcategories. Condition $(1)$ is certainly satisfied, since by hypothesis $(3)$ the quotients $\cat{S}/(\cat{S} \cap \cat{I}_a)$ are $\alpha$-compactly generated. For condition $(2)$ we must show that $\cat{S} \cap \cat{I}_a$ is $\alpha$-compactly generated in $\cat{S}/(\cat{S} \cap \cat{I}_b)$ for $\{a,b\} = \{1,2\}$. It follows from hypothesis $(4)$ that
\begin{equation}
(\cat{S} \cap \cat{I}_a)^\alpha = (\cat{S} \cap \cat{I}_a) \cap (\cat{T}/\cat{I}_b)^\alpha
\end{equation}
and from hypothesis $(3)$ that
\begin{equation}\label{eq:proof_of_theorem_5}
\Big( \cat{S}/(\cat{S} \cap \cat{I}_b) \Big)^\alpha = \Big( \cat{S}/(\cat{S} \cap \cat{I}_b) \Big) \cap (\cat{T}/\cat{I}_b)^\alpha.
\end{equation}
This implies that the inclusion $\cat{S} \cap \cat{I}_a \lto \cat{S}/(\cat{S} \cap \cat{I}_b)$ preserves $\alpha$-compactness, from which we deduce that the former category is $\alpha$-compactly generated in the latter. Now, using the first part of the theorem, we conclude that $\cat{S}$ is $\alpha$-compactly generated and that an object $X \in \cat{S}$ is $\alpha$-compact in $\cat{S}$ if and only if the image under $\cat{S} \lto \cat{S}/(\cat{S} \cap \cat{I}_a)$ is $\alpha$-compact for each $a \in \{1,2\}$. By (\ref{eq:proof_of_theorem_5}) these are precisely the $X \in \cat{S}$ that are $\alpha$-compact in $\cat{T}$, so $\cat{S}$ is $\alpha$-compactly generated in $\cat{T}$.
\end{proof}

\section{Proof of the Theorem}\label{section:proof}

Let us briefly recall the setup of Theorem \ref{theorem:main_theorem_general}. We are given a triangulated category $\cat{T}$ with coproducts, a regular cardinal $\alpha$, and a cocovering $\cat{F} = \{ \cat{I}_1,\ldots, \cat{I}_n \}$ satisfying some conditions $(1),(2)$, and we wish to prove that $\cat{T}$ is $\alpha$-compactly generated. Once again, since the $\alpha = \aleph_0$ case is handled by \cite[Theorem 5.15]{Rouquier08}, we restrict to the case $\alpha > \aleph_0$. In what follows we make implicit use of the properties of proper intersection described in Appendix \ref{section:proper_intersection}, particularly Lemma \ref{lemma:intersecting_bousfield}.

\begin{proof}[Proof of Theorem \ref{theorem:main_theorem_general}]
The proof is by induction on the number $n \ge 2$ of elements in the cocover $\cat{F}$ (to be clear, the induction includes the second statement of the theorem, about $\cat{S}$). The $n = 2$ case is given by Proposition \ref{prop:main_theorem_1}, and for the inductive step the argument is identical to the inductive step in the proof of \cite[Theorem 5.15]{Rouquier08}. For the reader's convenience, let us repeat the argument here. Assume that $n > 2$ and set
\[
\cat{I}_{\cap} = \cat{I}_2 \cap \cdots \cap \cat{I}_n.
\]
Then $\{ \cat{I}_1, \cat{I}_{\cap} \}$ is an orthogonal pair of Bousfield subcategories of $\cat{T}$. By hypothesis $\cat{T}/\cat{I}_1$ is $\alpha$-compactly generated, and $\cat{I}_{\cap}$ is $\alpha$-compactly generated in $\cat{T}/\cat{I}_{1}$, so in order to apply the $n = 2$ case of the Theorem to the pair $\cat{I}_1, \cat{I}_{\cap}$ it remains to check that
\begin{itemize}
\item[(i)] $\cat{T}/\cat{I}_{\cap}$ is $\alpha$-compactly generated, and
\item[(ii)] $\cat{I}_1$ is $\alpha$-compactly generated in $\cat{T}/\cat{I}_{\cap}$.
\end{itemize}
Set $\overline{\cat{T}} = \cat{T}/\cat{I}_{\cap}$ and for $\cat{I} \in \cat{F}$ define $\overline{\cat{I}} = \cat{I}/(\cat{I} \cap \cat{I}_{\cap})$. This is a Bousfield subcategory of $\overline{\cat{T}}$, and $\{ \overline{\cat{I}_2}, \ldots, \overline{\cat{I}_n} \}$ is a cocovering of $\overline{\cat{T}}$ (Lemma \ref{lemma:intersecting_bousfield}). Moreover:
\begin{itemize}
\item For $\cat{I} \in \cat{F} \setminus \{ \cat{I}_1 \}$ the category $\overline{\cat{T}}/\overline{\cat{I}} \cong \cat{T}/\cat{I}$ is $\alpha$-compactly generated.
\item For $\cat{I} \in \cat{F} \setminus \{ \cat{I}_1 \}$ and a nonempty subset $\cat{F}' \subseteq \cat{F} \setminus \{ \cat{I}, \cat{I}_1 \}$ the image of the canonical functor
\[
\bigcap_{\cat{I}' \in \cat{F}'} \overline{\cat{I}'} \xlto{\inc} \overline{\cat{T}} \xlto{\can} \overline{\cat{T}}/\overline{\cat{I}} \cong \cat{T}/\cat{I}
\]
is just the essential image of the composite
\[
\bigcap_{\cat{I}' \in \cat{F}'} \cat{I}' \xlto{\inc} \cat{T} \xlto{\can} \cat{T}/\cat{I},
\]
which is, by hypothesis, $\alpha$-compactly generated in $\cat{T}/\cat{I}$.
\end{itemize}
From the inductive hypothesis, we deduce that $\overline{\cat{T}}$ is $\alpha$-compactly generated, and that $X \in \overline{\cat{T}}$ is $\alpha$-compact if and only if the images of $X$ in $\overline{\cat{T}}/\overline{\cat{I}} \cong \cat{T}/\cat{I}$ are $\alpha$-compact for each $\cat{I} \in \cat{F} \setminus \{ \cat{I}_1 \}$. This verifies condition $(i)$ above, and it remains to check $(ii)$.

Identify $\cat{I}_1$ as a subcategory of $\overline{\cat{T}}$ via the embedding $\cat{I}_1 \lto \cat{T} \lto \cat{T}/\cat{I}_{\cap}$. Then $\cat{I}_1$ is a Bousfield subcategory, properly intersecting $\overline{\cat{I}}$ for $\cat{I} \in \cat{F} \setminus \{ \cat{I}_1 \}$. Moreover:
\begin{itemize}
\item For $\cat{I} \in \cat{F} \setminus \{ \cat{I}_1 \}$ the subcategory $\cat{I}_1/(\cat{I}_1 \cap \overline{\cat{I}})$ of $\overline{\cat{T}}/\overline{\cat{I}}$ is identified, under the equivalence $\overline{\cat{T}}/\overline{\cat{I}} \cong \cat{T}/\cat{I}$, with $\cat{I}_1/(\cat{I}_1 \cap \cat{I})$, which is $\alpha$-compactly generated in $\cat{T}/\cat{I}$ by hypothesis.
\item For every $\cat{I} \in \cat{F} \setminus \{ \cat{I}_1 \}$ and nonempty subset $\cat{F}' \subseteq \cat{F} \setminus \{ \cat{I}, \cat{I}_1 \}$ the image of
\[
\cat{I}_1 \cap \bigcap_{\cat{I}' \in \cat{F}'} \overline{\cat{I}'} \xlto{\inc} \overline{\cat{T}} \xlto{\can} \overline{\cat{T}}/\overline{\cat{I}} \cong \cat{T}/\cat{I}
\]
is just the essential image of the composite
\[
\cat{I}_1 \cap \bigcap_{\cat{I}' \in \cat{F}'} \cat{I}' \xlto{\inc} \cat{T} \xlto{\can} \cat{T}/\cat{I}
\]
which is, by hypothesis, $\alpha$-compactly generated in $\cat{T}/\cat{I}$.
\end{itemize}
From the inductive hypothesis (with $\cat{S} = \cat{I}_1$) we conclude that $\cat{I}_1$ is $\alpha$-compactly generated in $\overline{\cat{T}}$. Having now established both $(i)$ and $(ii)$ above, we deduce from the $n = 2$ case of the Theorem that $\cat{T}$ is $\alpha$-compactly generated, and that $X \in \cat{T}$ is $\alpha$-compact if and only if $X$ is $\alpha$-compact in both $\cat{T}/\cat{I}_1$ and $\overline{\cat{T}}$. But the image of $X$ in $\overline{\cat{T}}$ is $\alpha$-compact if and only if the images of $X$ in $\overline{\cat{T}}/\overline{\cat{I}} \cong \cat{T}/\cat{I}$ are $\alpha$-compact for $\cat{I} \in \cat{F} \setminus \{ \cat{I}_1 \}$, which gives the desired criterion for $\alpha$-compactness in $\cat{T}$.

To complete the inductive step, it remains to treat the second statement: we are given a Bousfield subcategory $\cat{S}$ properly intersecting every $\cat{I} \in \cat{F}$, satisfying conditions $(3),(4)$. By hypothesis, then, $\cat{S}/(\cat{S} \cap \cat{I}_1)$ and $\cat{I}_{\cap} \cap \cat{S}$ are $\alpha$-compactly generated in $\cat{T}/\cat{I}_1$, and to apply the $n = 2$ case of the Theorem to $\cat{S}$ and the cocover $\{ \cat{I}_1, \cat{I}_{\cap} \}$ it remains to check that
\begin{itemize}
\item[(i)'] $\cat{S}/(\cat{S} \cap \cat{I}_{\cap})$ is $\alpha$-compactly generated in $\overline{\cat{T}}$, and
\item[(ii)'] $\cat{I}_1 \cap \cat{S}$ is $\alpha$-compactly generated in $\overline{\cat{T}}$.
\end{itemize}
Set $\overline{\cat{S}} = \cat{S}/(\cat{S} \cap \cat{I}_{\cap})$. This is a Bousfield subcategory of $\overline{\cat{T}}$ properly intersecting every element of the cocovering $\{\overline{\cat{I}_2},\ldots,\overline{\cat{I}_n}\}$ of $\overline{\cat{T}}$. Moreover:
\begin{itemize}
\item For $\cat{I} \in \cat{F} \setminus \{ \cat{I}_1 \}$ the subcategory $\overline{\cat{S}}/(\overline{\cat{S}} \cap \overline{\cat{I}})$ of $\overline{\cat{T}}/\overline{\cat{I}}$ is identified under the equivalence $\overline{\cat{T}}/\overline{\cat{I}} \cong \cat{T}/\cat{I}$ with the subcategory $\cat{S}/(\cat{S} \cap \cat{I})$, which is $\alpha$-compactly generated in $\cat{T}/\cat{I}$ by hypothesis.
\item For every $\cat{I} \in \cat{F} \setminus \{ \cat{I}_1 \}$ and nonempty subset $\cat{F}' \subseteq \cat{F} \setminus \{ \cat{I}, \cat{I}_1 \}$ the image of
\[
\overline{\cat{S}} \cap \bigcap_{\cat{I}' \in \cat{F}'} \overline{\cat{I}'} \xlto{\inc} \overline{\cat{T}} \xlto{\can} \overline{\cat{T}}/\overline{\cat{I}} \cong \cat{T}/\cat{I}
\]
is just the essential image of the composite
\[
\cat{S} \cap \bigcap_{\cat{I}' \in \cat{F}'} \cat{I}' \xlto{\inc} \cat{T} \xlto{\can} \cat{T}/\cat{I},
\]
which is, by hypothesis, $\alpha$-compactly generated in $\cat{T}/\cat{I}$.
\end{itemize}
From the inductive hypothesis, we conclude that $\overline{\cat{S}}$ is $\alpha$-compactly generated in $\overline{\cat{T}}$, which is $(i)'$ above. A similar argument verifies $(ii)'$, and from the $n = 2$ case of the Theorem we conclude that $\cat{S}$ is $\alpha$-compactly generated in $\cat{T}$. This completes the inductive step, and thus the proof.
\end{proof}

\begin{corollary}\label{corollary:alpha_compactness_local} In the situation of Theorem \ref{theorem:main_theorem_general}, for any regular cardinal $\beta \ge \alpha$ an object $X \in \cat{T}$ is $\beta$-compact if and only if the image of $X$ is $\beta$-compact in $\cat{T}/\cat{I}$ for every $\cat{I} \in \cat{F}$.
\end{corollary}
\begin{proof}
If a triangulated category $\cat{Q}$ is $\alpha$-compactly generated, or a subcategory $\cat{S}$ is $\alpha$-compactly generated in some larger triangulated category, then the same is true for any regular cardinal $\beta \ge \alpha$. Hence, if the cocover $\cat{F}$ satisfies the hypotheses of Theorem \ref{theorem:main_theorem_general} for $\alpha$, it satisfies the same conditions for $\beta \ge \alpha$, whence the claim.
\end{proof}

\section{Derived Categories of Rings}\label{section:derived_cat_rings}

In the next section we obtain a characterisation of the $\alpha$-compact objects in the derived category of a scheme. We will use a reduction to the affine case, so in this section we prepare the ground with a review of some facts about the derived category of a ring. Throughout, a \emph{ring} is a (not necessarily commutative) associative ring with identity, and all modules are left modules. Given a ring $R$ we denote by $\qder{R}$ the unbounded derived category of $R$-modules. If $\alpha$ is a regular cardinal, then $\qder{R}^\alpha$ denotes the full subcategory of $\alpha$-compact objects in $\qder{R}$.

A complex of $R$-modules $P$ is called \emph{$\K$-projective} if, for every acyclic complex $X$ of $R$-modules, the complex of abelian groups $\Hom_R(P, X)$ is acyclic \cite{Spalt88}. For example, any bounded above complex of projective $R$-modules is $\K$-projective. The \emph{$\K$-projective resolution} of a complex of $R$-modules $M$ is a quasi-isomorphism $P \lto M$, where $P$ is $\K$-projective. In this case, $P$ is the \emph{unique} (up to homotopy equivalence) $\K$-projective complex isomorphic to $M$ in $\qder{R}$.

\begin{theorem}[(Neeman)]\label{theorem:neeman_compacts} Let $R$ be a ring. The derived category $\qder{R}$ is compactly generated, and given a regular cardinal $\alpha > \aleph_0$ a complex of $R$-modules is $\alpha$-compact in $\qder{R}$ if and only if it is quasi-isomorphic to a $\K$-projective complex of free $R$-modules of rank $< \alpha$.
\end{theorem}

\begin{remark} Since we are allowing noncommutative rings $R$, one has to be a bit careful about the meaning of ``rank''; we direct the reader to \cite[(5.2)]{Neeman08}.
\end{remark}

\begin{proof}
Part of this criterion is stated without proof in \cite{NeemanBook}, and the full statement can be deduced from the more general results of \cite[\S 7]{Neeman08}. To be precise: taking $\K$-projective resolutions defines a fully faithful functor $\qder{R} \lto \kproj{R}$, and we identify $\qder{R}$ as a subcategory of $\kproj{R}$ via this embedding. Neeman proves in \cite{Neeman08} that $\kproj{R}$ is $\aleph_1$-compactly generated. Since $\qder{R}$ is a localising subcategory of $\kproj{R}$ generated by $R$, which is compact in $\kproj{R}$, it follows from Theorem \ref{theorem:thomason_neeman} that for any regular cardinal $\alpha > \aleph_0$ we have
\[
\qder{R}^\alpha = \kproj{R}^\alpha \cap \qder{R},
\]
that is, a complex $M$ of $R$-modules is $\alpha$-compact in $\qder{R}$ if and only if the $\K$-projective resolution $P$ of $M$ is $\alpha$-compact in $\kproj{R}$. But by \cite[Proposition 7.4]{Neeman08} and \cite[Proposition 7.5]{Neeman08}, $P$ is $\alpha$-compact in $\kproj{R}$ if and only if it is homotopy-equivalent to a complex of free $R$-modules of rank $< \alpha$.
\end{proof}

\begin{remark}\label{remark_subtlety} In the context of the theorem, it is natural to ask if the condition of $\K$-projectivity can be dropped, that is: are the $\alpha$-compact objects in $\qder{R}$ precisely the complexes quasi-isomorphic to a complex of free modules of rank $< \alpha$? We will see in Theorem \ref{theorem:neeman_compact_cohomology} that this is true provided that $R$ is either left noetherian or has cardinality $< \alpha$. In general, however, the answer is negative: we construct a counterexample in Appendix \ref{section:counterexample}, consisting of a ring $B$ and a complex of free $B$-modules of rank $< \alpha$ that is not $\alpha$-compact in $\qder{B}$.
\end{remark}

\begin{definition} Let $R$ be a ring and $\alpha$ an infinite cardinal. An $R$-module $M$ is said to be $\alpha$-\emph{generated} if it can be generated as an $R$-module by a subset of cardinality $\alpha$. We say that $M$ is \emph{sub-$\alpha$-generated} if it can be generated by a subset of cardinality $< \alpha$.
\end{definition}

We include a proof of the following standard fact:

\begin{lemma}\label{lemma:generation_property} Let $R$ be a ring, $\alpha$ an infinite cardinal, and suppose that $R$ is either left noetherian, or has cardinality $\le \alpha$. Then if an $R$-module $M$ is $\alpha$-generated, every submodule of $M$ is $\alpha$-generated.
\end{lemma}
\begin{proof}
Let $\kappa$ be a fixed regular cardinal with $\kappa > \alpha$. We prove the following statement by transfinite induction: if $x$ is an ordinal $< \kappa$, and $M$ is any $R$-module generated by a set of cardinality $|x|$, then every submodule of $M$ is generated by a set of cardinality $< \kappa$. Call this statement $B(x)$. Taking $x = \alpha$ and $\kappa$ to be the successor cardinal of $\alpha$ (which is regular) gives the result.

Successor ordinals: if $x$ is an ordinal $< \kappa$ then either $|x| = |x^+|$, in which case the statement $B(x^+)$ is just $B(x)$ and the inductive step is trivial, or these two cardinals are distinct, in which case $x$ is a finite cardinal. If $x$ is finite, then since $R$ is either left noetherian or has cardinality $\le \alpha$, it is straightforward to verify that $B(x)$ holds.

Limit ordinals: assume that $x$ is a limit ordinal $< \kappa$, and that $B(\beta)$ holds for all $\beta < x$. We may assume that $x$ is a cardinal, since otherwise the inductive step is trivial. Let $M$ be an $R$-module generated by a set of cardinality $x$, say by $\{ m_t \l t < x \}$, and for $t < x$ let $M_{<t}$ be the submodule of $M$ generated by the set $\{ m_s \l s < t \}$. This is a generating set of cardinality $< x$, so by the inductive hypothesis for any submodule $N$ of $M$, the intersection $N \cap M_{< t}$ can be generated by a set $\lambda_t$ of cardinality $< \kappa$. From the equality
\[
N = N \cap M = N \cap \sum_{t < x} M_{< t} = \sum_{t < x} (N \cap M_{<t })
\]
we deduce that $N$ can be generated by the union $\lambda = \cup_{t < x} \lambda_t$, which is of cardinality 
\[
|\lambda| \le \sum_{t < x} |\lambda_t| < \kappa
\]
because $\kappa$ is regular.
\end{proof}

The following argument was kindly explained to the author by Neeman:

\begin{theorem}[(Neeman)]\label{theorem:neeman_compact_cohomology} Let $R$ be a ring and $\alpha > \aleph_0$ a regular cardinal. Suppose that $R$ is either left noetherian, or has cardinality $< \alpha$. Then for a complex $F$ of $R$-modules the following are equivalent:
\begin{itemize}
\item[(i)] $F$ is $\alpha$-compact in $\qder{R}$.
\item[(ii)] $F$ is isomorphic, in $\qder{R}$, to a complex of free $R$-modules of rank $< \alpha$.
\item[(iii)] $H^i(F)$ is a sub-$\alpha$-generated $R$-module for all $i \in \mathbb{Z}$.
\end{itemize}
\end{theorem} 
\begin{proof}
The implication $(i) \Rightarrow (ii)$ is a consequence of Theorem \ref{theorem:neeman_compacts}, while $(ii) \Rightarrow (iii)$ follows from Lemma \ref{lemma:generation_property}, so it remains to prove that $(iii) \Rightarrow (i)$. Let $\qderl{\alpha}{R}$ denote the full subcategory of $\qder{R}$ consisting of complexes with sub-$\alpha$-generated cohomology. Using Lemma \ref{lemma:generation_property} this is easily checked to be an $\alpha$-localising triangulated subcategory. We already know that $\qder{R}^\alpha \subseteq \qderl{\alpha}{R}$, and we want to prove the reverse inclusion. Let us begin with a technical observation.

Given a complex $F$ in $\qderl{\alpha}{R}$, we claim that it is possible to construct a morphism of complexes $\phi: F \lto F'$ with $F' \in \qderl{\alpha}{R}$ such that the mapping cone of $\phi$ is $\alpha$-compact and $\phi$ is a \emph{ghost}, that is, the induced maps $H^i(\phi): H^i(F) \lto H^i(F')$ are zero for all $i \in \mathbb{Z}$.

For each $i \in \mathbb{Z}$ there exists a surjective map $P^i \lto H^i(F)$ from some free $R$-module $P^i$ of rank $< \alpha$, and this lifts to a morphism of complexes $g_i: \Sigma^{-i} P^i \lto F$. The coproduct $P = \bigoplus_{i \in \mathbb{Z}} \Sigma^{-i} P^i$ is $\alpha$-compact in $\qder{R}$, and the sum $g: P \lto F$ of the $g_i$'s is a morphism of complexes surjective on cohomology. Extending to a triangle
\[
P \xlto{g} F \xlto{\phi} F' \lto \Sigma P
\]
we observe that $\phi$ is a ghost with $\alpha$-compact mapping one, and $F'$ has sub-$\alpha$-generated cohomology. If we begin with a fixed $F = F_0$ in $\qderl{\alpha}{R}$ and iterate this process to construct a sequence
\[
F_0 \xlto{\phi_0} F_1 \xlto{\phi_1} F_2 \xlto{\phi_2} \cdots
\]
of ghost maps $\phi_i$ in $\qderl{\alpha}{R}$ with $\alpha$-compact cones, then the homotopy colimit of this sequence in $\qder{R}$ is acyclic (that is, zero). Since $\qderl{\alpha}{R}$ is closed under countable coproducts, this also calculates the homotopy colimit in $\qderl{\alpha}{R}$, and thus in the quotient $\qderl{\alpha}{R}/\qder{R}^\alpha$. But in this quotient each $\phi_i$ is an isomorphism (as it has zero cone), so in this case the homotopy colimit is equal to $F$. Together, these observations imply that $F \in \qder{R}^\alpha$, as claimed.
\end{proof}


\begin{remark} We learn from the theorem that when $\alpha > |R|$, the $\alpha$-compacts in $\qder{R}$ can be characterised by the ``size'' of their cohomology. In fact, this is a general phenomenon in well-generated triangulated categories, as explained by Krause in \cite[Theorem B]{Krause01} and \cite[Theorem C]{Krause02}. 
\end{remark}

\section{Derived Categories of Schemes}\label{example:compacts_derived}

In this section we apply Theorem \ref{theorem:main_theorem_general} to the derived category of quasi-coherent sheaves on a scheme. We refer the reader to \cite{LipmanNotes, Neeman96, Tarrio97, Tarrio08} for background on unbounded derived categories of schemes. A scheme is \emph{semi-separated} if it admits a covering by affine open subsets $\{ V_i \}_{i \in I}$ with all pairwise intersections $V_i \cap V_j$ affine; see \cite{Tarrio08, ThomasonTrobaugh}. Separated schemes are semi-separated, and semi-separated schemes are quasi-separated. Given a scheme $X$ we denote by $\qder{X}$ the unbounded derived category of quasi-coherent sheaves on $X$. If $X$ is quasi-compact and semi-separated, then $\qder{X}$ is equivalent to the full subcategory of complexes with quasi-coherent cohomology in the derived category of sheaves of $\cat{O}_X$-modules \cite{BokNee93}.
 
Let $X$ be a quasi-compact semi-separated scheme, and $\{ U_1, \ldots, U_n \}$ a cover of $X$ by affine open subsets. For $1 \le i \le n$ set $Z_i = X \setminus U_i$ and denote by $\qderl{Z_i}{X}$ the full subcategory of $\qder{X}$ consisting of complexes with cohomology supported on $Z_i$. The inclusion $\qderl{Z_i}{X} \lto \qder{X}$ and restriction $(-)|_{U_i}$ fit into a sequence of functors
\[
\xymatrix@C+1pc{
0 \ar[r] & \qderl{Z_i}{X} \ar[r]^{\inc} & \qder{X} \ar[r]^{(-)|_{U_i}} & \qder{U_i} \ar[r] & 0
}
\]
which is exact, in the sense that $(-)|_{U_i}$ induces an equivalence $\qder{X}/\qderl{Z_i}{X} \xlto{\sim} \qder{U_i}$. In fact this is a localisation sequence: the right adjoint of $\qderl{Z_i}{X} \lto \qder{X}$ is Grothendieck's local cohomology functor $\rdev\Gamma_{Z_i}(-)$, and the right adjoint of $(-)|_{U_i}$ is the derived direct image $\rdev f_*$, where $f: U_i \lto X$ is the inclusion. In particular, each $\qderl{Z_i}{X}$ is a Bousfield subcategory of $\qder{X}$.

The family $\cat{F} = \{ \qderl{Z_1}{X}, \ldots, \qderl{Z_n}{X} \}$ is a cocovering of $\qder{X}$, and Rouquier proves in \cite[\S 6.2]{Rouquier08} that this cocovering satisfies the hypotheses $(1),(2)$ of Theorem \ref{theorem:main_theorem_general} for the cardinal $\alpha = \aleph_0$. Let us examine the content of these hypotheses, and sketch Rouquier's argument in each case:
\begin{itemize}
\item[(1)] requires that $\qder{X}/\qderl{Z_i}{X} \cong \qder{U_i}$ be compactly generated for $1 \le i \le n$. But $U_i \cong \Spec(A_i)$ is affine, and thus $\qder{U_i} \cong \qder{A_i}$ is known to be compactly generated.
\item[(2)] requires, given an index $1 \le j \le n$ and nonempty subset $I \subseteq \{ 1,\ldots, n \}$ not containing $j$, that the essential image of the composite
\[
\xymatrix@C+1pc{
\bigcap_{i \in I} \qderl{Z_i}{X} \ar[r]^(0.6){\inc} & \qder{X} \ar[r]^{(-)|_{U_j}}& \qder{U_j}
}
\]
be compactly generated in $\qder{U_j}$. But this image is just $\qderl{Z}{U_j}$, where $Z$ is the complement of $U_j \cap \bigcup_{i \in I} U_i$ in $U_j$. Since $U_j$ is affine one can generate this subcategory by a Koszul complex, which is compact in $\qder{U_j}$; see \cite[Proposition 6.6]{Rouquier08} or \cite[Proposition 6.1]{BokNee93}.
\end{itemize}
In particular, Rouquier proves that if $Z$ is a closed subset of $X$ with quasi-compact complement $U$, then $\qderl{Z}{X}$ is compactly generated in $\qder{X}$. Since the restriction functor $(-)|_U: \qder{X} \lto \qder{U}$ factors as $\qder{X} \lto \qder{X}/\qderl{Z}{X} \cong \qder{U}$ we infer from \cite[Theorem 4.4.9]{NeemanBook} that the functor $(-)|_U$ preserves $\alpha$-compactness for any regular cardinal $\alpha$. 
\vspace{0.2cm}

Let us record the following special case of Corollary $\ref{corollary:alpha_compactness_local}$, with $\cat{T} = \qder{X}$ and $\cat{F}$ as above.

\begin{proposition}\label{prop:alpha_compactness_local_derived} Let $X$ be a quasi-compact semi-separated scheme and $\alpha$ a regular cardinal. Given a cover $\{ U_1, \ldots, U_n \}$ of $X$ by affine open subsets, a complex $\md{F}$ of quasi-coherent sheaves on $X$ is $\alpha$-compact in $\qder{X}$ if and only if $\md{F}|_{U_i}$ is $\alpha$-compact in $\qder{U_i}$ for $1 \le i \le n$.
\end{proposition}

The proposition reduces the problem of understanding the $\alpha$-compacts in $\qder{X}$ to the problem of understanding the $\alpha$-compacts in the derived category of a ring, which was settled in Section \ref{section:derived_cat_rings}.

\begin{proposition} Let $X$ be a quasi-compact semi-separated scheme, and $\alpha > \aleph_0$ a regular cardinal. A complex $\md{F}$ of quasi-coherent sheaves on $X$ is $\alpha$-compact in $\qder{X}$ if and only if, for every $x \in X$, there is an affine open neighborhood $U$ of $x$ such that $\Gamma(U, \md{F})$ is quasi-isomorphic as a complex of $A = \Gamma(U, \cat{O}_X)$-modules to a $\K$-projective complex of free $A$-modules of rank $< \alpha$.
\end{proposition}
\begin{proof}
Using Proposition \ref{prop:alpha_compactness_local_derived} we reduce to the case of affine $X$, which is Theorem \ref{theorem:neeman_compacts}.
\end{proof}

Let $\qder{X}^\alpha$ denote the full subcategory of $\alpha$-compact objects in $\qder{X}$. Subcategories of $\qder{X}$ are typically defined by imposing conditions on homology, so it is comforting to have such a description of $\qder{X}^\alpha$. We begin with some definitions. 

\begin{definition} Let $\alpha$ be an infinite cardinal. A quasi-coherent sheaf $\md{F}$ on a scheme $X$ is said to be \emph{locally $\alpha$-generated} if for every $x \in X$, there exists an open neighborhood $U$ of $x$ together with an epimorphism $\bigoplus_{j \in J} \cat{O}_X|_U \lto \md{F}|_U$, for some index set $J$ of cardinality $\alpha$. If for each $x \in X$ we can arrange for the set $J$ to be of cardinality $< \alpha$, then $\md{F}$ is \emph{locally sub-$\alpha$-generated}.
\end{definition}

\begin{lemma}\label{lemma:technical_generation} Let $R$ be a commutative ring, $\alpha$ an infinite cardinal and $F$ an $R$-module. Then the following are equivalent:
\begin{itemize}
\item[(i)] $F$ is an $\alpha$-generated $R$-module.
\item[(ii)] The complex $\md{F}$ of quasi-coherent sheaves on $\Spec(R)$ associated to $F$ is locally $\alpha$-generated.
\end{itemize}
\end{lemma}
\begin{proof}
$(i) \Rightarrow (ii)$ is clear. For $(ii) \Rightarrow (i)$, set $X = \Spec(R)$ and suppose that $\md{F}$ is locally $\alpha$-generated. We may find generators $f_1,\ldots,f_r$ of the unit ideal of $R$, with the property that on each $U_i = D(f_i)$ there is an epimorphism $\oplus_{j \in J_i} \cat{O}_X|_{U_i} \lto \md{F}|_{U_i}$ for some set $J_i$ of cardinality $\alpha$. That is, $F[f_i^{-1}]$ can be generated as an $R[f_i^{-1}]$-module by a subset $\{ a_j/f_i^{n_j} \}_{j \in J_i}$ of cardinality $|J_i| = \alpha$. Form a set $J$ consisting of the union, over each $1 \le i \le r$, of the set of numerators $\{ a_j \}_{j \in J_i}$. This set has cardinality $\alpha$ and generates $F$ as an $R$-module.
\end{proof}

Finally, we arrive at a characterisation of $\alpha$-compactness in terms of cohomology sheaves.

\begin{corollary} Let $X$ be a quasi-compact semi-separated scheme and $\alpha > \aleph_0$ a regular cardinal. Suppose that $X$ is either
\begin{itemize}
\item[(a)] noetherian, or
\item[(b)] admits a cover by open affines $\{ U_i \}_{1 \le i \le n}$ with $\Gamma(U_i, \cat{O}_X)$ of cardinality $< \alpha$ for $1 \le i \le n$.
\end{itemize}
Then a complex $\md{F}$ of quasi-coherent sheaves on $X$ is $\alpha$-compact in $\qder{X}$ if and only if $H^i(\md{F})$ is locally sub-$\alpha$-generated for every $i \in \mathbb{Z}$.
\end{corollary}
\begin{proof}
Under either hypothesis on $X$ there is an affine open cover $\{ U_1, \ldots, U_n \}$  of $X$ such that the conclusion of Theorem \ref{theorem:neeman_compact_cohomology} applies to each of the rings $\Gamma(U_i, \cat{O}_X)$. By Lemma \ref{lemma:technical_generation}, a quasi-coherent sheaf $\md{G}$ on $X$ is locally sub-$\alpha$-generated if and only if $\Gamma(U_i, \md{G})$ is a sub-$\alpha$-generated $\Gamma(U_i, \cat{O}_X)$-module for $1 \le i \le n$, so the result follows from Proposition \ref{prop:alpha_compactness_local_derived}.
\end{proof}

\appendix

\section{Counterexample}\label{section:counterexample}

Let $R$ be a ring and $\alpha > \aleph_0$ a regular cardinal. We proved in Theorem \ref{theorem:neeman_compact_cohomology} that when $R$ is either left noetherian or has cardinality $< \alpha$, the $\alpha$-compact objects in $\qder{R}$ are the objects isomorphic to a complex of free $R$-modules of rank $< \alpha$. In this appendix we show that this characterisation of the $\alpha$-compact objects cannot hold for arbitrary rings, by constructing for any given regular cardinal $\alpha > \aleph_0$ a commutative local ring $B$ and a complex of free $B$-modules of rank one which is not $\alpha$-compact. Needless to say, $B$ is non-noetherian and has cardinality $\ge \alpha$.

Let $k$ be a field and $\beta$ an infinite cardinal, and $I = \{ x_i \}_{i \in \beta}$ a set of variables indexed by $\beta$. We denote by $k[[I]]$ the ring of formal power series in the set of variables $I$. More precisely, let $\mathbb{N}^{(I)}$ be the set of all functions $\gamma: I \lto \mathbb{N}$ with finite support, and define
\[
k[[ I ]] := \{ \text{functions $f: \mathbb{N}^{(I)} \lto k$} \},
\]
with the usual addition $(f + g)(\gamma) = f(\gamma) + g(\gamma)$ and product $(f \cdot g)(\gamma) = \sum_{\alpha + \beta = \gamma} f(\alpha) g(\beta)$. Then $k[[I]]$ is a commutative local ring, with maximal ideal given by the ideal of power series with zero constant term. We say that a monomial in the set of variables $I$ is \emph{pure} if it is of the form $x_i^n$ for some $i \in \beta$ and $n \ge 0$, with $x_i^0$ understood to be the identity in $k[[I]]$. Let $\mf{a}$ denote the ideal of power series in which the coefficient of every pure monomial is zero (e.g. $x_ix_j$ for $i \neq j$) and define
\[
B' := k[[I]]/\mf{a}.
\]
Each residue class of $B'$ contains a unique power series $f$ in which only pure monomials have nonzero coefficients, and such $f$ can be written as a formal sum
\begin{equation}\label{eq:formal_power_series}
f = \lambda \cdot 1 + \sum_{i \in \beta} \sum_{n \ge 1} f_{i, n} \cdot x_i^n, \qquad \lambda, f_{i,n} \in k.
\end{equation}
We say that a power series $f$ \emph{involves} a variable $x_i$ if the coefficient of $x_i^n$ in $f$ is nonzero for some $n \ge 1$. Finally, we may define the desired ring $B$ as a subring of $B'$.

\begin{definition} Given a field $k$ and an infinite cardinal $\beta$, we define a commutative $k$-algebra $B$ by
\[
B := \{ f \in B' \l \text{$f$ involves only finitely many variables in $I$} \}.
\]
Concretely, this is the ring of formal power series of the form (\ref{eq:formal_power_series}) in the set of variables $\{ x_i \}_{i \in \beta}$, with each power series involving only a finite number of variables. Power series are multiplied according to the relations $x_i^n x_i^m = x_i^{n+m}$ and $x_i x_j = 0$ for $i \neq j$. This is a commutative local ring, with maximal ideal $\mf{m}_B$ given by the set of power series with zero constant term, and residue field $k = B/\mf{m}_B$.
\end{definition}

\begin{remark} The following properties of $B$ are immediate:
\begin{itemize}
\item Given a nonempty subset $J \subseteq \beta$, the ideal $(x_j)_{j \in J}$ in $B$ consists precisely of those power series $f$ with zero constant term, which do not involve any variable $x_i$ with $i \in \beta \setminus J$.
\item Given $i \in \beta$, the kernel of $B \xlto{x_i} B$ is the ideal $(x_j)_{j \in \beta \setminus \{ i \}}$.
\item For any nonempty subset $J \subseteq \beta$ there is an internal direct sum $(x_j)_{j \in J} \cong \bigoplus_{j \in J} (x_j)$.
\end{itemize}
\end{remark}

We will need the following consequence of Theorem \ref{theorem:neeman_compacts}.

\begin{lemma}\label{lemma:compact_implies_rank_limit} Let $A$ be a commutative local ring with residue field $k$, and $\alpha > \aleph_0$ a regular cardinal. If $M$ is an $A$-module belonging to $\qder{A}^{\alpha}$ then $\textup{rank}_k \textup{Tor}_n( M, k ) < \alpha$ for all $n \ge 0$.
\end{lemma}
\begin{proof}
If $M$ belongs to $\qder{A}^\alpha$ then by Theorem \ref{theorem:neeman_compacts} it admits a $\K$-projective resolution by a complex $P$ of free $A$-modules of rank $< \alpha$. Hence $P \otimes_A k = M \otimes^{\bold{L}}_A k$ is a complex of $k$-vector spaces of dimension $< \alpha$, and the claim follows.
\end{proof}

The key pathology of the ring $B$ becomes apparent in the next lemma.

\begin{lemma}\label{lemma:step_to_contradiction} Given $i \in \beta$ and the corresponding ideal $(x_i)$ in $B$, we have
\[
\textup{rank}_k \textup{Tor}_n( (x_i), k) = \begin{cases} 1 & n = 0,\\ \beta & n > 0. \end{cases}
\]
Consequently, if $\alpha > \aleph_0$ is a regular cardinal such that $\alpha \le \beta$, then $(x_i)$ does not belong to $\qder{B}^\alpha$.
\end{lemma}
\begin{proof}
To construct a free resolution of $(x_i)$ we begin with the epimorphism $B \xlto{x_i} (x_i)$, which has kernel $\oplus_{i_1 \in \beta \setminus \{i\}} (x_{i_1})$. One proceeds to construct the following resolution of $(x_i)$, call it $F$:
\[
\cdots \lto \bigoplus_{\substack{i_1 \in \beta\setminus\{i\}\\ i_2 \in \beta \setminus \{i_1\}\\ i_3 \in \beta \setminus \{i_2\}}}\!\!\!\! B \lto \bigoplus_{\substack{i_1 \in \beta \setminus \{i\}\\ i_2 \in \beta \setminus \{i_1\}}}\!\!\!\! B \lto \bigoplus_{i_1 \in \beta \setminus \{i\}} \!\!\! B \lto B
\]
with modules $F^{-n} = \bigoplus_{i_1 \in \beta \setminus \{i\}, \ldots, i_n \in \beta \setminus \{i_{n-1}\}} B$ for $n \ge 1$ (we set $i_0 = i$) and differentials
\[
\partial^{-n}_F = \bigoplus_{i_1 \in \beta \setminus \{i\}, \ldots, i_{n-1} \in \beta \setminus \{ i_{n-2} \}} (x_{i_n})_{i_n \in \beta \setminus \{i_{n-1}\}},
\]
where $(x_{i_n})_{i_n \in \beta \setminus \{i_{n-1}\}}$ denotes an infinite row matrix $\oplus_{i_n \in \beta \setminus \{i_{n-1}\}} B \lto B$. Notice that for $n \ge 1$ the module $F^{-n}$ is free of rank $\beta$, and the differentials in $F$ are annihilated by $- \otimes_A k$, so the vector space $\textup{Tor}_n((x_i), k) = H^{-n}( F \otimes_A k)$ has the desired rank for $n \ge 0$. It follows from Lemma \ref{lemma:compact_implies_rank_limit} that $(x_i)$ is not $\alpha$-compact in $\qder{B}$.
\end{proof}

\begin{proposition} Let $\alpha > \aleph_0$ be a regular cardinal and suppose that $\alpha \le \beta$. Then the complex
\[
T: \cdots \lto 0 \lto 0 \lto B \xlto{x_0} B \xlto{x_1} B \xlto{x_0} B \xlto{x_1} \cdots
\]
does not belong to $\qder{B}^\alpha$.
\end{proposition}
\begin{proof}
To be clear, the complex $T$ is zero in degrees $< 0$ and $B$ in degrees $\ge 0$, with differentials given by alternating products with $x_0$ and $x_1$. It is not difficult to see that there is a quasi-isomorphism
\[
(x_i)_{i > 0} \oplus \bigoplus_{k \ge 1} \Sigma^{-k} (x_i)_{i > 1} \cong T,
\]
and an isomorphism of $B$-modules $(x_i)_{i > 0} \cong \bigoplus_{i > 0} (x_i)$, so that for any $i > 0$ in $\beta$ the ideal $(x_i)$ is a direct summand of $T$ in $\qder{B}$. If $T$ belonged to $\qder{B}^\alpha$ then, since this subcategory is thick, we would have $(x_i) \in \qder{B}^\alpha$. But this possibility is excluded by Lemma \ref{lemma:step_to_contradiction}, whence $T \notin \qder{B}^\alpha$.
\end{proof}

To conclude, if $\alpha > \aleph_0$ is a regular cardinal, and we take any field $k$ and $\beta = \alpha$ in the above, we obtain a commutative local ring $B$ and complex $T$ of free $B$-modules of rank one, with $T \notin \qder{B}^\alpha$.

\section{Proper Intersection}\label{section:proper_intersection}

Throughout this appendix, $\cat{T}$ is a triangulated category. We say that a subcategory $\cat{C}$ of $\cat{T}$ is \emph{strictly full} if it is full, and whenever $X \in \cat{T}$ is isomorphic to an object of $\cat{C}$, then $X \in \cat{C}$.

\begin{definition} Let $\cat{A}, \cat{B}$ be strictly full subcategories of $\cat{T}$ closed under $\Sigma, \Sigma^{-1}$. Then $\cat{A} \star \cat{B}$ denotes the full subcategory of $\cat{T}$ given by the objects $X \in \cat{T}$ for which there exists a triangle
\[
A \lto X \lto B \lto \Sigma A
\]
with $A \in \cat{A}$ and $B \in \cat{B}$. This is a strictly full subcategory closed under $\Sigma, \Sigma^{-1}$.
\end{definition}

\begin{remark} Let $\cat{A}, \cat{B}, \cat{C}$ be strictly full subcategories closed under $\Sigma, \Sigma^{-1}$. It is a simple exercise in the octahedral axiom to see that $\cat{A} \star (\cat{B} \star \cat{C}) = (\cat{A} \star \cat{B}) \star \cat{C}$. Clearly if $\cat{A}$ is a triangulated subcategory, then $\cat{A} \star \cat{A} = \cat{A}$.
\end{remark}

\begin{definition}\label{defn:strictly_intersecting} Two triangulated subcategories $\cat{A}, \cat{B}$ of $\cat{T}$ are said to \emph{intersect properly} if there is an equality of subcategories $\cat{A} \star \cat{B} = \cat{B} \star \cat{A}$.
\end{definition}

In the next lemma we verify that this definition of proper intersection agrees with the one given by Rouquier \cite[(5.2.3)]{Rouquier08} for Bousfield subcategories.

\begin{lemma} Let $\cat{A}, \cat{B}$ be triangulated subcategories of $\cat{T}$. The following are equivalent:
\begin{itemize}
\item[(i)] $\cat{A}$ and $\cat{B}$ intersect properly, that is, $\cat{A} \star \cat{B} = \cat{B} \star \cat{A}$.
\item[(ii)] Given $A \in \cat{A}$ and $B \in \cat{B}$, every morphism $A \lto B$ and every morphism $B \lto A$ factors through an object of $\cat{A} \cap \cat{B}$.
\end{itemize}
\end{lemma}
\begin{proof}
$(i) \Rightarrow (ii)$ Let $f: A \lto B$ be given, with $A \in \cat{A}$ and $B \in \cat{B}$, and extend to a triangle
\[
A \xlto{f} B \lto X \lto \Sigma A.
\]
Since $X$ belongs to $\cat{B} \star \cat{A} = \cat{A} \star \cat{B}$, there is a triangle $A' \lto X \lto B' \lto \Sigma A'$ with $A' \in \cat{A}$ and $B' \in \cat{B}$. Applying the octahedral axiom to the pair $B \lto X, X \lto B'$ we obtain an object $D$ and triangles
\begin{gather*}
B \lto B' \lto D \xlto{\gamma} \Sigma B,\\
A \xlto{\delta} \Sigma^{-1} D \lto A' \lto \Sigma A,
\end{gather*}
such that $\gamma \circ \Sigma \delta = \Sigma f$. From the first triangle we deduce that $D \in \cat{B}$, and from the second triangle we conclude that $D \in \cat{A}$, whence $f$ factors via $\cat{A} \cap \cat{B}$. The factorisation argument for a morphism $B \lto A$ is dual.
$(ii) \Rightarrow (i)$ Let $X \in \cat{A} \star \cat{B}$ be given, so that there is a triangle
\[
A \lto X \lto B \xlto{s} \Sigma A. 
\]
By hypothesis $s$ factors as $B \lto D \lto \Sigma A$ for some $D \in \cat{A} \cap \cat{B}$, and from the octahedral axiom applied to the pair of morphisms in this factorisation of $s$, we conclude that $X \in \cat{B} \star \cat{A}$. This shows that $\cat{A} \star \cat{B} \subseteq \cat{B} \star \cat{A}$, and the reverse inclusion follows similarly.
\end{proof}

\begin{lemma}\label{lemma:perp_int_triangles} Let $\cat{A}, \cat{B}$ be properly intersecting triangulated subcategories of $\cat{T}$. Then $\cat{A} \star \cat{B}$ is a triangulated subcategory of $\cat{T}$.
\end{lemma}
\begin{proof}
It suffices to prove that $\cat{A} \star \cat{B}$ is closed under mapping cones. Let $f: X \lto Y$ be a morphism in $\cat{T}$ with $X,Y \in \cat{A} \star \cat{B}$, and fix triangles
\begin{gather*}
X \xlto{f} Y \lto C \lto \Sigma X,\\
A_X \lto X \lto B_X \lto \Sigma A_X,\\
A_Y \lto Y \lto B_Y \lto \Sigma A_Y,
\end{gather*}
where $A_X,A_Y \in \cat{A}$ and $B_X, B_Y \in \cat{B}$. Applying the octahedral axiom to the pair $A_Y \lto Y, Y \lto C$ yields an object $D$ and triangles
\begin{gather}
A_Y \lto C \lto D \lto \Sigma A_Y,\label{eq:perp_int_triangles_1}\\
B_Y \lto D \lto \Sigma X \lto \Sigma B_Y.\label{eq:perp_int_triangles_2}
\end{gather}
Using the proper intersection property, we infer from (\ref{eq:perp_int_triangles_2}) that
\[
D \in \cat{B} \star (\cat{A} \star \cat{B}) = \cat{B} \star (\cat{B} \star \cat{A}) = (\cat{B} \star \cat{B}) \star \cat{A} = \cat{B} \star \cat{A} = \cat{A} \star \cat{B},
\]
whence by (\ref{eq:perp_int_triangles_1}) we have
\[
C \in \cat{A} \star (\cat{A} \star \cat{B}) = (\cat{A} \star \cat{A}) \star \cat{B} = \cat{A} \star \cat{B}.
\]
Hence $\cat{A} \star \cat{B}$ is closed under mapping cones, and therefore triangulated.
\end{proof}

\begin{remark} Let $\cat{A}, \cat{B}$ be properly intersecting triangulated subcategories of $\cat{T}$. Then $\cat{A} \star \cat{B}$ is clearly the smallest triangulated subcategory of $\cat{T}$ containing $\cat{A} \cup \cat{B}$. Notice that if $\cat{A}$ and $\cat{B}$ are localising, then so is $\cat{A} \star \cat{B}$.
\end{remark}

We will need the following results from \cite{Rouquier08}.

\begin{lemma} Let $\cat{A}, \cat{B}$ be properly intersecting Bousfield subcategories of $\cat{T}$. Then $\cat{A} \cap \cat{B}$ and $\cat{A} \star \cat{B}$ are Bousfield subcategories of $\cat{T}$.
\end{lemma}
\begin{proof}
See \cite[Lemma 5.8]{Rouquier08}.
\end{proof}

\begin{lemma}\label{lemma:intersecting_bousfield} Let $\cat{F}$ be a finite family of Bousfield subcategories of $\cat{T}$, any two of which intersect properly. Given a subset $\cat{F}' \subseteq \cat{F}$, the intersection $\cap_{\cat{I} \in \cat{F}'} \cat{I}$ is a Bousfield subcategory of $\cat{T}$ intersecting properly with any subcategory in $\cat{F}$. Given $\cat{I}_1,\cat{I}_2,\cat{I} \in \cat{F}$, the quotients $\cat{I}_1/(\cat{I}_1 \cap \cat{I})$ and $\cat{I}_2/(\cat{I}_2 \cap \cat{I})$ are properly intersecting Bousfield subcategories of $\cat{T}/\cat{I}$.
\end{lemma}
\begin{proof}
See \cite[Lemma 5.9]{Rouquier08}.
\end{proof}

\bibliographystyle{amsalpha}

\end{document}